\documentclass{amsart}

\usepackage{amssymb,amsmath,amsthm}
\usepackage{enumerate}
\usepackage{ifpdf}
\usepackage{url,hyperref}

\ifpdf
\usepackage[all,pdf]{xy} 
\else
\input xy
\xyoption{all}
\xyoption{2cell}
\xyoption{v2}
\fi

\newcommand{\Cat}{\textbf{\textrm{Cat}}}
\newcommand{\Gpd}{\textbf{\textrm{Gpd}}}
\newcommand{\Bicat}{\textbf{\textrm{Bicat}}}
\newcommand{\Sch}{\textbf{\textrm{Sch}}}

\newcommand{\id}{\mathrm{id}}
\newcommand{\into}{\hookrightarrow}
\newcommand{\equals}{\hspace{10pt}=\hspace{10pt}}

\newtheorem{theorem}{Theorem}[section]
\newtheorem{lemma}[theorem]{Lemma}
\newtheorem{proposition}[theorem]{Proposition}
\newtheorem{corollary}[theorem]{Corollary}
\theoremstyle{definition}
\newtheorem{definition}[theorem]{Definition}
\newtheorem{example}[theorem]{Example}
\newtheorem{remark}[theorem]{Remark}

\numberwithin{equation}{section}

\begin{document}

\title[2-categories admitting bicategories of fractions]%
{On certain 2-categories admitting localisation by bicategories of fractions}

\date{}

\author[D.~M.~Roberts]{David Michael Roberts}
\address{School of Mathematical Sciences, University of Adelaide, SA 5005,
Australia}
\email{david.roberts@adelaide.edu.au}
\thanks{DMR was supported financially by ARC grant number DP120100106, and emotionally by Mrs~R. 
This document is released under a CC0 license \url{http://creativecommons.org/publicdomain/zero/1.0/}}

\begin{abstract}

	Pronk's theorem on bicategories of fractions is applied, in almost all cases in the literature, to 2-categories of geometrically presentable stacks on a 1-site. 
	We give an proof that subsumes all previous such results and which is purely 2-categorical in nature, ignoring the nature of the objects involved.
	The proof holds for 2-categories that are not (2,1)-categories, and we give conditions for local essential smallness.


\end{abstract}


\maketitle

\section{Introduction}

The area of \emph{higher geometry} deals broadly with generalisations of `spaces', be they topological, differential geometric, algebro-geometric etc.,
that can be represented by groupoids (or higher groupoids) in the original category of spaces. 
Usually these go by the label differential, topological, algebraic etc.~stacks, but when viewed as stacks there are more morphisms between objects than when viewed simply as internal groupoids; there are non-invertible maps of groupoids that become equivalences of the associated stacks. 
Pronk, in \cite{Pronk_96}, formulated what it meant to localise a bicategory at a class of morphisms and introduced a bicategory of fractions that exists under certain conditions in order to construct this localisation.
She then went on to show that 2-categories of differentiable, topological and algebraic stacks (of certain sorts) were indeed localisations of the 2-categories of groupoids internal to the appropriate categories.

Since then, many other cases of 2-categorical localisations have been considered, using Pronk's result applied to other categories (for extensive discussion and examples see \cite[\S\S2,8]{Roberts_12}). 
However, almost all of them---only two exceptions are known to the author---deal with internal groupoids and/or stacks in some setting. 
In this case, the 2-category in question, and the class of morphisms at which one wants to localise, satisfy some properties making available a much simpler calculus of fractions, namely \emph{anafunctors}. 
These were introduced by Makkai \cite{Makkai} for the category of sets \emph{sans} Choice and in the general internal setting by Bartels \cite{Bartels}. 
The author's \cite{Roberts_12} considered the case of a sub-2-category $C\into \Cat(S)$ of the 2-category of categories internal to a subcanonical site $(S,J)$, satisfying some mild closure conditions.
The main result of \cite{Roberts_12} is that such 2-categories admit a bicategory of fractions at the so-called \emph{weak equivalences} (also called \emph{Morita equivalences}), and that anafunctors also calculate this localisation.

This note serves to show that given a 2-category with the structure of a \emph{2-site} of a certain form (all covering maps must be representably fully faithful), the same result holds -- namely that the bicategory of fractions of Pronk exists.
One can then approach the theory of presentable stacks (on 1-sites) in a formal way, analogous to Street's formal theory of stacks \cite{Street_82} (cf Shulman's \cite{Shulman_12}). 
This result covers all others in the literature dealing with localising 2-categories of internal categories or groupoids. 
It may also replicate the result in \cite{Pronk-Warren_14}, although the framework therein is conceptually more pleasing; the theorems of this note are definitely sufficient to imply the applications of the abstract framework of \cite{Pronk-Warren_14} 

Both \cite{Pronk-Warren_14}, and the recent paper \cite{Abbad-Vitale_13} (written in parallel with the present note), deal with constructing localisations via fibrancy/projectivity. 
Hom-categories in the constructions of localisations in both papers are in fact hom-categories of the original bicategory, and so one is assured of local smallness, a problem when localising any large (bi-)category, using local smallness of the original bicategory.
The present note does not assume existence of enough fibrant objects or projectives to prove local (essential) smallness (\ref{prop:loc_ess_small}).
It certainly assumes less than the applications in \cite{Pronk-Warren_14} (prestacks on a subcanonical site) or \cite{Abbad-Vitale_13} (internal groupoids in a regular category).

Sometimes when calculating the localisation of a 2-category of internal groupoids, various authors use what are variously known as \emph{Hilsum-Skandalis morphisms} or \emph{right principal bibundles} (see \cite[\S 2]{Roberts_12} for discussion and references). 
In the more general setting of 2-sites as defined here such a definition is not possible, as one has hom-categories that are not groupoids.
Additionally, composition of 1-arrows in the bicategory of internal groupoids and bibundles requires existence of pullback-stable reflexive coequalisers, an assumption not made here. 
Also, the definition of a bibundle between internal \emph{categories} is not clear and the right notion of a map of bibundles (i.e.~2-arrows in the localisation) does not appear to be as simple as in the groupoid case.

The author thanks the organisers of the Australian Category Seminar for the opportunity to present an early version this work in October 2011.
Comments by the referee lead to a rethink of this paper and subsequent strengthening of the results.

\section{Preliminaries}

Though this paper touches lightly on the theory of bicategories, a knowledge of 2-categories is sufficient (an accessible reference is \cite{Leinster_98}).
We consider our 2-categories to have one extra piece of structure, namely an analogue of a Grothendieck pretopology.

\begin{definition}
	\label{def:strict_pretopology}

	A \emph{fully faithful singleton coverage} on a 2-category $K$ is a class $J$ of 1-arrows satisfying the following properties:

	\begin{enumerate}[(i)]

		\item
			$J$ contains the identity arrows and is closed under composition;

      \item
			for all $q\colon u\to x\in J$ and 1-arrows $f\colon y\to x$, there is a square
			\[
               \xymatrix{
					v \ar[d]_k \ar[r] & u\ar[d]^q_{\ }="s" \\
					y \ar[r]_{f}^{\ }="t" & x
					\ar@{=>}_{\simeq}"s";"t"
				}
			\]
			with $k\in J$;

      \item
			for any $q\colon u\to x$ in $J$ the functor $q_*\colon K(z,u)\to K(z,x)$ is fully faithful;
	\end{enumerate}
	Morphisms satisfying (iii) are called \emph{ff 1-arrows}. A \emph{2-site} will here denote a 2-category equipped with a fully faithful singleton coverage.

\end{definition}

For brevity this paper will use the terminology \emph{2-site} even though this has been used elsewhere for something more general. 
Note that that $K$ is not necessarily small, but in what follows may sometimes be locally essentially small. 
That is, the hom-categories $K(x,y)$ are equivalent to small categories for all objects $x$ and $y$.

One might think about the 1-arrows in $J$ as being something like acyclic fibrations in a category with fibrant objects, without the requirement for the existence of fibrant objects.

We define the analogue of weak equivalences in this setting.

\begin{definition}
	\label{def:weak_equiv}

	A 1-arrow $x\to y$ in $(K,J)$ is called \emph{$J$-locally split} if there is a map $u\to y$ in $J$ and a diagram of the form
	\[
		\xymatrix{
		& x\ar[d] \\
		u \ar[r]^(.7){\ }="t" \ar[ur]_(.7){\ }="s" & y
		\ar@{=>}"s";"t"
		}
	\]
	with the 2-arrow an isomorphism. 
	A \emph{weak equivalence} in $(K,J)$ is an ff 1-arrow that is $J$-locally split. 
	The class of weak equivalences will be denoted $W_J$.

\end{definition}

\begin{example}\label{eg:internal_cats}

	As an example, take the 2-category $K$ to be $\Cat(S)$ or $\Gpd(S)$ for $(S,T)$ a finitely complete site with singleton pretopology $T$.\footnote{Recall that a singleton pretopology $T$ is a class of arrows containing all identity arrows and closed under composition and pullbacks (which must exist for arrows in $T$).} 
	One can also take the 2-category of Lie groupoids, which is course is not finitely complete -- in this instance, $T$ can be taken as the pretopology of surjective submersions.
	In each instance the pretopology $J=J(T)$ is defined to be the class of fully faithful functors such that the object component is a cover in $T$.
	Then $K$ is a 2-site, as one can take pullbacks of 1-arrows in $J$, and fully faithful functors are closed under pullback.
	It is an easy result \cite[lemma 4.13]{Roberts_12} that the resulting weak equivalences in the sense of definition \ref{def:weak_equiv} are the same as weak equivalences between internal categories in the sense of Bunge-Par\'e \cite{Bunge-Pare_79}. 

\end{example}

We shall need a more general definition for use later.

\begin{definition}
	\label{def:cofinal}
	Let $A$ be a class of 1-arrows in a 2-category, and $A'$ a subclass.
	We say $A'$ is \emph{cofinal} in $A$ if for every $f\colon x\to y$ in $A$, there is a $g\colon z\to y$ in $A'$ and an $s\colon z\to x$ such that $f\circ s \simeq g$.
	If for every object $y$, the arrows in $A'$ with codomain $y$ comprise a set, we say $A'$ is a \emph{locally small} cofinal class.

\end{definition}

Thus $J$ is cofinal in $W_J$, but we will later use classes $J' \subset J$ that do not give the structure of a 2-site as above.

Given a 2-category (or bicategory) $B$ with a class $W$ of 1-arrows, we say that a 2-functor $Q\colon B \to \widetilde{B}$ is a \emph{localisation of $B$ at $W$} if it sends the 1-arrows in $W$ to equivalences in $\widetilde{B}$ and is universal with this property.
This latter means that for any bicategory $A$
precomposition with $Q$,
\[
	Q^* \colon \Bicat(\widetilde{B},A) \to \Bicat_W(B,A),
\]
is an equivalence of hom-bicategories, with $\Bicat_W$ meaning the full sub-bicategory on those 2-functors sending arrows in $W$ to equivalences. 
The definition of a bicategory of fractions of \cite{Pronk_96} gives a reasonably convenient way to calculate the localisation at a class of arrows, satisfying properties as follows:

\begin{enumerate}[BF1]
       \item
            $W$ contains all equivalences;
       \item
            $W$ is closed under composition and isomorphism;
       \item
            for all $w\colon a' \to a,\ f\colon c \to a$ with $w\in W$ there
            exists a 2-commutative square
               \[
               \xymatrix{
					p \ar[d]_v \ar[r] & a'\ar[d]^w_{\ }="s" \\
					c \ar[r]_f^{\ }="t" & a
					\ar@{=>}_{\simeq}"s";"t"
				}
				\]
               with $v\in W$;
       \item

            if $\alpha\colon w \circ f \Rightarrow w \circ g$ is a 2-arrow and $w\in W$ there is a 1-cell $v \in W$ and a 2-arrow $\beta\colon f\circ v \Rightarrow g \circ v$ such that $\alpha\circ v = w \circ \beta$.

			\noindent Moreover:
			when $\alpha$ is an isomorphism, we require $\beta$ to be an isomorphism too; when $v'$ and $\beta'$ form another such pair, there exist 1-cells $u,\,u'$ such that $v\circ u$ and $v'\circ u'$ are in $W$, and an isomorphism $\epsilon\colon v\circ u \Rightarrow v' \circ u'$ such that the following diagram commutes:

               \begin{equation}\label{2cf4.diag}
                    \xymatrix{
                           f \circ v \circ u \ar@{=>}[r]^{\beta\circ u}
                           \ar@{=>}[d]_{f\circ \epsilon}^\simeq &
                           g\circ v \circ u
                           \ar@{=>}[d]^{g\circ \epsilon}_\simeq \\
                           f\circ v' \circ u' \ar@{=>}[r]_{\beta'\circ u'} &
                           g\circ v' \circ u'
                   }
               \end{equation}
\end{enumerate}
If BF1--BF4 hold, we say $(B,W)$ \emph{admits a bicategory of fractions}.
Given such a pair $(B,W)$, Pronk constructed a new bicategory $B[W^{-1}]$ with the same objects as $B$ and a functor $U\colon B \to B[W^{-1}]$ that is a localisation of $B$ at $W$. 
We will describe the (underlying graphs of the) hom-categories of $B[W^{-1}]$, since this is the most detail we need for the results below.

Let $x$ and $y$ be objects of $B[W^{-1}]$ (which are just objects of $B$).
The 1-arrows from $x$ to $y$ are spans
\[
	x \xleftarrow{w} u \xrightarrow{f} y
\]
where $w\in W$. 
The 2-arrows $(w_1,f_1) \Rightarrow (w_2,f_2)$ are represented by diagrams
\[
	\xymatrix{
		& u_1 \ar[dl]_{w_1}^(.6){\ }="s1" \ar[dr]^{f_1}_(.6){\ }="s2"\\
		x & v \ar[u]_{p_1} \ar[d]^{p_2} & y\\
		& u_2 \ar[ul]^{w_2}_(.6){\ }="t1" \ar[ur]_{f_2}^(.6){\ }="t2"
		\ar@{=>}"s1";"t1"^\alpha
		\ar@{=>}"s2";"t2"_\beta
	}
\]
where $w_i\circ p_i$ is in $W$ for $i=1,2$ and $\alpha$ is invertible.
Two such diagrams, with data $(v,p_1,p_2,\alpha,\beta)$ and $(v',p'_1,p'_2,\alpha',\beta')$, are equivalent when there exists a diagram
\[
	\xymatrix{
		u_1 & \ar[l]_{p'_1} v'\\
		v \ar[u]^{p_1} \ar[d]_{p_2} & 
		t 	\ar[l]_{q}="t1"^{\ }="s2" 
			\ar[u]_{q'}^(.6){\ }="s1"
			\ar[d]^{q'}_(.6){\ }="t2" \\
		u_2 & \ar[l]^{p'_2}_{\ } v'
		\ar@{=>}"s1";"t1"_{\gamma_1}
		\ar@{=>}"s2";"t2"_{\gamma_2}
	}
\]
where $\gamma_1$ and $\gamma_2$ are invertible, $w_1\circ p_1\circ q$ and $w_1\circ p'_1\circ q'$ are in $W$ and 
\[
\noindent\makebox[\textwidth]{
	\raisebox{36pt}{
		\xymatrix{
			& \ar[dl]_{w_1}^{\ }="s3" u_1 & \ar[l]_{p'_1} v'\\
			x &v \ar[u]_{p_1} \ar[d]^{p_2} & 
			t 	\ar[l]_{q}="t1"^{\ }="s2" 
				\ar[u]_{q'}^(.6){\ }="s1"
				\ar[d]^{q'}_(.6){\ }="t2" \\
			& \ar[ul]^{w_2}_{\ }="t3" u_2 & \ar[l]^{p'_2}_{\ } v'
			\ar@{=>}"s1";"t1"_{\gamma_1}
			\ar@{=>}"s2";"t2"_{\gamma_2}
			\ar@{=>}"s3";"t3"_\alpha
		}
	}
	\; = \;
	\raisebox{36pt}{
		\xymatrix{
			& \ar[dl]_{w_1}^{\ }="s" u_1 \\
			x & v' \ar[u]_{p'_1} \ar[d]^{p'_2} & \ar[l]_{q'} t\;,\\
			& u_2 \ar[ul]^{w_2}_{\ }="t"
			\ar@{=>}"s";"t"_{\alpha'}
		}
	}
	\quad
	\raisebox{36pt}{
		\xymatrix{
			 v'  \ar[r]^{p'_1} & u_1 \ar[dr]^{f_1}_{\ }="s3" \\
			t \ar[u]^{q'}_(.6){\ }="s1" \ar[d]_{q'}^(.6){\ }="t2" \ar[r]^{q}="t1"_{\ }="s2" & 
			v 	\ar[u]^{p_1}^(.6){\ }
				\ar[d]_{p_2}_(.6){\ }& y\\
			 v' \ar[r]_{p'_2}^{\ } & u_2 \ar[ur]_{f_2}^{\ }="t3"
			\ar@{=>}"s1";"t1"^{\gamma_1}
			\ar@{=>}"s2";"t2"^{\gamma_2}
			\ar@{=>}"s3";"t3"^\beta
		}
	}
	\; = \;
	\raisebox{36pt}{
		\xymatrix{
			 & u_1 \ar[dr]^{f_1}_{\ }="s3" \\
			t  \ar[r]^{q'}="t1"_{\ }="s2" & 
			v' 	\ar[u]^{p'_1}^(.6){\ }
				\ar[d]_{p'_2}_(.6){\ }& y\; .\\
			  & u_2 \ar[ur]_{f_2}^{\ }="t3"
			\ar@{=>}"s3";"t3"^{\beta'}
		}
	}
}\]
We then define a 2-arrow of $B[W^{-1}]$ to be an equivalence class of diagrams.

The following lemma is stated in more generality in \cite{Tommasini_I}, but we shall merely state it in terms that we need here. 

\begin{lemma}[Tommasini, \cite{Tommasini_I}, Lemma 6.1]
	\label{Tommasini_lemma}

	Let $(K,W)$ be a 2-category admitting a bicategory of fractions.
	Let $F_i = (x\xleftarrow{w_i} u_i \xrightarrow{f_i} y)$, $i=1,2$ be 1-arrows in $K[W^{-1}]$, and 
	\[
       \xymatrix{
			v \ar[d]_{p_2} \ar[r]^{p_1} & u_1\ar[d]^{w_1}_{\ }="s" \\
			u_2 \ar[r]_{w_2}^{\ }="t" & x
			\ar@{=>}_{\alpha}"s";"t"
		}
	\]
	be a chosen filler, using BF3, such that $w_i\circ p_i\in W$, $i=1,2$. Then any 2-arrow $F_1 \Rightarrow F_2$ in $K[W^{-1}]$ is represented by a diagram of the form
	\[
		\xymatrix{
			& u_1 \ar[dl]^{\ }="s"  & \ar@{=}[l] u_1 \ar[dr]_{\ }="s2"^{f_1}\\
			x & v \ar[u]_{p_1} \ar[d]^{p_2}& v' \ar[l]_{q'} \ar[u] \ar[d]& y\\
			&u_2 \ar[ul]_{\ }="t"  & \ar@{=}[l] u_2 \ar[ur]^{\ }="t2"_{f_2}
			\ar@{=>}"s";"t"_\alpha
			\ar@{=>}"s2";"t2"^\beta
		}
	\]
	where $q'\in W$.

\end{lemma}

The conclusion of the lemma in \cite{Tommasini_I} does not mention that $q'\in W$, but examination of the proof shows that it is so.

\section{Results}

The first main result is as follows:

\begin{theorem}
	\label{thm:2-sites_have_fractions}

	A 2-site $(K,J)$ admits a bicategory of fractions for $W_J$.

\end{theorem}

\begin{proof} 
We verify the conditions in the definition of a bicategory of fractions.

\begin{enumerate}[BF1]
	\item
	An internal equivalence $f\colon x\to y$ is clearly $J$-locally split. 
	Let $g\colon y\to x$ be a pseudoinverse to $f$, and let $w$ be some object of $K$. 
	Then $g_*$ is a pseudoinverse to $f_*$, where $f_*\colon K(w,x)\to K(w,y)$ is post-composition with $f$ (and analogously with $g_*$). 
	But then it is a well-known fact that equivalences of categories are fully faithful, and so $f$ is a ff 1-arrow.

	\item
	That the composition of ff 1-arrows is again ff, and that ff 1-arrows are closed under isomorphism follows from the analogous fact for fully faithful functors between categories. 
	So we only need to show the same for $J$-locally split arrows. 
	Consider the composition $g\circ f$ of two $J$-locally split arrows,
	\[
		\xymatrix{
			u \ar[d] \ar@/^.5pc/[dr]_{\ }="s1"^(.35){q}&v \ar[d]^s
			\ar@/^.5pc/[dr]_(.5){\ }="s2"^(.4){p}& \\
			x\ar[r]_{f}^(.33){\ }="t1" & y \ar[r]_{g}^(.33){\ }="t2" & z
			\ar@{=>}"s1";"t1"
			\ar@{=>}"s2";"t2"
		}
	\]
	The cospan $u\xrightarrow{q}y\xleftarrow{s} v$ completes to a 2-commuting square with top arrow $w \to v$ in $J$. 
	The composite $w \to z$ is in $J$, all 2-arrows are invertible, hence $g\circ f$ is $J$-locally split.

	Let $w,f\colon x\to y$ be 1-arrows, $w$ be $J$-locally split and $a\colon w \Rightarrow f$ invertible. 
	It is immediate from the diagram
	\[
		\xymatrix{
			u \ar[dd] \ar@/^.7pc/[ddrr]_{\ }="s1"^{u} \\
			\\
			x\ar@/^1pc/[rr]^{w}="t1"_{\ }="s2" \ar@/_1pc/[rr]_{f}^{\ }="t2"
			&&y
			\ar@{=>}"s1";"t1"
			\ar@{=>}"s2";"t2"^{a}
		}
	\]
	that $f$ is also $J$-locally split.

	\item
	Let $w\colon x\to y$ be a weak equivalence, and let $f\colon z\to y$ be any other 1-arrow. 
	From the definition of $J$-locally split, we have the diagram
	\[
		\xymatrix{
			u \ar[d] \ar@/^.5pc/[dr]_{\ }="s1"^{q}& \\
			x\ar[r]_{w}^(.33){\ }="t1"&y
			\ar@{=>}"s1";"t1"
		}
	\]
	We complete the cospan to get a 2-commuting diagram
	
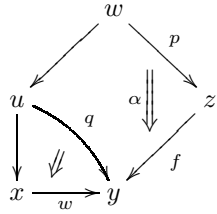
\[
		\xymatrix{
			& w \ar[dr]^{p}_{\ }="s2" \ar[dl] \\
			u \ar[d] \ar@/^.5pc/[dr]_{\ }="s1"^{q}& & z \ar[dl]^f_{\ }="t2"\\
			x\ar[r]_{w}^(.33){\ }="t1" & y
			\ar@{=>}"s1";"t1"
			\ar@{=>}"s2";"t2"_\alpha
		}
	\]
	with $p\in J$, $\alpha$ invertible, and by the trivial observation $J\subset W_J$, we have $p\in W_J$ and a 2-commuting square as required.

	\item
	Since $J$-equivalences are ff, given
	
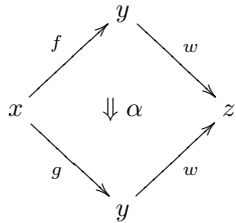
\[
	       \xymatrix{
	               &y \ar[dr]^w \\
	               x \ar[ur]^f \ar[dr]_g & \Downarrow \alpha & z\\
	               & y \ar[ur]_w
	       }
	\]
	where $w\in W_J$, there is a unique $\beta\colon f \Rightarrow g$ such that
	\[
	        \raisebox{36pt}
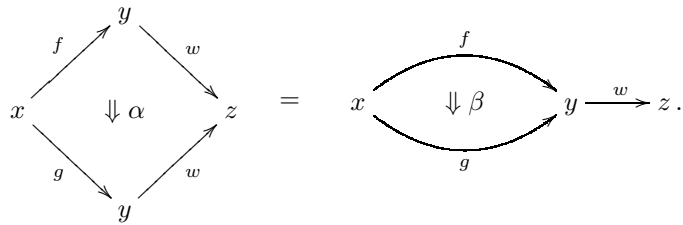
{
	        \xymatrix{
	               &y \ar[dr]^w \\
	               x \ar[ur]^f \ar[dr]_g & \Downarrow \alpha & z\\
	               & y \ar[ur]_w
	       }
	       }
	       \equals
	        \raisebox{36pt}{
	       \xymatrix{
		       &&\\
		       x \ar@/^1.5pc/[rr]^f \ar@/_1.5pc/[rr]_g&
		       \Downarrow \beta & y \ar[r]^w & z \,.
	       }
	       }
	\]
	The existence of $\beta$ is the first half of BF4, with $v=\id_x$.
	Note that if $\alpha$ is an isomorphism, so is $\beta$, since $w$ is ff.
	Given $v'\colon t\to x \in W_J$ such that there is a 2-arrow
	\[
	       \xymatrix{
	               &x \ar[dr]^f \\
	               t \ar[ur]^{v'} \ar[dr]_{v'} & \Downarrow \beta' & y\\
	               & x \ar[ur]_g
	       }
	\]
	satisfying
	\begin{align*}
	        \raisebox{36pt}{
	       \xymatrix{
	               & x \ar[dr]^f \\
	               w \ar[ur]^{v'} \ar[dr]_{v'} & \Downarrow \beta' & y \ar[r]^w & z\\
	               & x \ar[ur]_g
	       }
	       }
	       \equals &
	        \raisebox{36pt}{
	        \xymatrix{
	               && y \ar[dr]^w \\
	               t \ar[r]^{v'} & x \ar[ur]^f \ar[dr]_g & \Downarrow \alpha & z\\
	               && y \ar[ur]_w
	       }
	        }      \nonumber \\
	        \equals &
	        \raisebox{36pt}{
	        \xymatrix{
	               &&\\
	               t\ar[r]^{v'}& x \ar@/^1.5pc/[rr]^f \ar@/_1.5pc/[rr]_g &\Downarrow \beta &
	               y \ar[r]^w & z
	       }
	        }\, ,
	\end{align*}
	then the fact $w$ is ff gives us
	\[
	        \raisebox{36pt}{
	       \xymatrix{
	               & x \ar[dr]^f \\
	               t \ar[ur]^{v'} \ar[dr]_{v'} & \Downarrow \beta' & y \\
	               & x \ar[ur]_g
	       }
	       }
	        \equals
	        \raisebox{36pt}{
	        \xymatrix{
	               &&\\
	               t\ar[r]^{v'}&x \ar@/^1.5pc/[rr]^f \ar@/_1.5pc/[rr]_g
	               &\Downarrow \beta
	               & y
	       }
	        }\, .
	\]
	This is precisely the diagram (\ref{2cf4.diag}) with $v=\id_x$, $u=v'$, $u'=\id_w$ and $\epsilon$ the identity 2-arrow. 
	Hence BF4 holds.
\end{enumerate}
\end{proof}

This theorem should be compared with the theorem in the paper \cite{Abbad-Vitale_13} (written in independently and in parallel with the present work). 
The authors show there that given a class $\Sigma$ of ff arrows in a bicategory satisfying certain conditions, there is a bicategory of fractions for $\Sigma$.
The class of arrows $W_J$ satisfies the conditions for a \emph{faithful calculus of fractions} \cite[Definition 2.4]{Abbad-Vitale_13}, using similar arguments as the preceeding proof.
The characterisation of $W_J$ as arising from a class $J$ as in definition \ref{def:strict_pretopology} is a means to arrive at a multitude of examples.

One would like to know if the localisation of $K$ at the weak equivalences is locally essentially small.

\begin{proposition}
	\label{prop:loc_ess_small}
	
	Let $K$ be a locally essentially small 2-category with a class of 1-arrows $W$ satisfying BF1--BF4. 
	If there is a locally small cofinal set $V \subset W$ then $K[W^{-1}]$ is locally essentially small.

\end{proposition}

\begin{proof}
	First given any fraction $x \xleftarrow{w} u \to y$ (with $w\in W_J$) giving a 1-arrow in Pronk's construction \cite[\S\S 2.2, 2.3]{Pronk_96} of $K[W^{-1}]$, there is an isomorphic 1-arrow $x\xleftarrow{v} t\to y$ where $v\in V$.
	Thus there are only set-many choices of backwards-pointing arrows with which to form fractions, and by local essential smallness of $K$, only set-many fractions from $x$ to $y$ up to isomorphism.

	To show the hom-categories $K[W^{-1}](x,y)$ are locally small, given a pair of fractions $x \xleftarrow{w_i} u_i \xrightarrow{f_i} y$, choose a filler for the cospan $u_1 \to x \leftarrow u_2$, namely
	\[
		\xymatrix{
			v \ar[r]^{p_2} \ar[d]_{p_1} & u_2 \ar[d]^{w_2}_{\ }="s"\\
			u_1 \ar[r]_{w_1}^{\ }="t" & x\;.
			\ar@{=>}"s";"t"_\alpha
		}
	\]
	Then arrows in the hom-category are given, by lemma \ref{Tommasini_lemma}, by the data of an arrow $q\colon v'\to v \in W$ and a 2-arrow $\beta\colon f_1\circ p_1 \circ q \Rightarrow f_2 \circ p_2 \circ q$ in $K$. 
	Fixing $q$, there are only set-many arrows as shown, since $K$ is locally essentially small. 
	Hence if we show a 2-arrow given by $(q,\beta)$ is also given (using the equivalence relation defining the 2-arrows of $K[W^{-1}]$) by $(r,\gamma)$ where $r\in V$, the hom-category is locally small.
	Assume we have an arrow given by the data $(q,\beta)$, and we have a second $r\in W$ and a diagram
	\[
		\xymatrix{
			& v'\ar[d]^q_{\ }="s"\\
			v_0 \ar[r]_r^{\ }="t" \ar[ur]^s & v
			\ar@{=>}"s";"t"_\phi
		}
	\]
	with $\phi$ invertible. 
	Defining $\gamma$ as
	\[
		\xymatrix{
			& v \ar[r] & u_1 \ar[dr]_(.3){\ }="s2" \\
			v_0 \ar[ur]^r_(.6){\ }="s1" \ar[dr]_r^(.6){\ }="t3" \ar[r]_(.9){\ }="s3"^(.9){\ }="t1" & v' \ar[u] \ar[d]&& y\\
			& v \ar[r] & u_2 \ar[ur]^(.3){\ }="t2"
			\ar@{=>}"s1";"t1"_{\phi^{-1}}
			\ar@{=>}"s2";"t2"_\beta
			\ar@{=>}"s3";"t3"_\phi
		}
	\]
	then one can use the span $v_0 \xleftarrow{=} v_0 \xrightarrow{s} v'$ and the 2-arrows $\phi$ and $\phi^{-1}$ to show that $(r,\gamma)$ defines the same 2-cell of $K[W^{-1}]$ as $(q,\beta)$.
\end{proof}

\begin{corollary}
	Given a 2-site $(K,J)$, if $K$ is locally essentially small and  $J$ has a locally small cofinal set, then $K[W_J^{-1}]$ is locally essentially small.
\end{corollary}

Notice that local essential smallness in \emph{not} automatic. 

\begin{example}
	Take the category $\Sch$ of schemes (over a base scheme, if one likes).
	The singleton pretopology of \emph{fpqc} maps (see e.g.\ \cite[\S 2.3]{Vistoli_05}) is such that $(\Gpd(\Sch),J(\text{\emph{fpqc}}))$ does not satisfy the hypotheses of proposition \ref{prop:loc_ess_small}.
	This is because the class of \emph{fpqc} maps has no locally small cofinal set \cite[\href{http://stacks.math.columbia.edu/tag/022A}{Tag 022A}]{stacks-project}.
\end{example}

There are even categories that are \emph{a priori} even better behaved in which proposition \ref{prop:loc_ess_small} may fail to hold; for example toposes that are well-pointed and that have a natural number object.
Such toposes are models for set theory without the Axiom of Choice, and categories internal to them are simply small categories.
Examples are given by the categories of sets in models of ZF as given by Gitik (cf \cite{vdBerg-Moerdijk_14}) and Karagila \cite{Karagila_14}, or the topos constructed in the author's \cite{Roberts_15}.

There are many examples of 2-sites to which the results of this note apply, for instance \cite{Roberts_12} spends five pages discussing some of them. 
This paper will add one example that is not covered by the results of \cite{Roberts_12}.

Let $S$ be a finitely complete category with a singleton pretopology $T$.
Consider the 2-category $\Cat(S)$ of categories internal to $S$, with the structure of a 2-site given by example \ref{eg:internal_cats}.
Note that 2-sites of this form are the ones considered in the many examples in \cite{Roberts_12}.

Now fix a category $X$ in $S$, and consider the lax slice 2-category $\Cat(S)/_lX$.
This 2-category has as objects functors $Z \to X$, 1-arrows triangles
\[
	\xymatrix{
		Z_1 \ar[rr]^f_(.6){\ }="s" \ar[dr]^(.7){\ }="t" && Z_2 \ar[dl]\\
		& X
		\ar@{=>}"s";"t"^a
	}
\]
and 2-arrows $(f_1,a_1) \Rightarrow (f_2,a_2)$ given by $f_1\Rightarrow f_2 \colon Z_1 \to Z_2$, commuting with $a_1$ and $a_2$. 
Let $J_X$ be the class of 1-arrows $(w,a)$ in $\Cat(S)/_lX$ such that $w\in J$, for $J$ as given in the previous paragraph, and $a$ is invertible. 
Likewise we have the slice 2-category $\Cat(S)/X$, where arrows $(f,a)$ have $a$ invertible, and the slice 2-category $\Gpd(S)/X$ where $X$ is an internal groupoid. 
These three examples are locally small 2-categories when $S$ is a locally small category.

\begin{proposition}

	The class of arrows $J_X$ makes $(\Cat(S)/_lX,J_X)$ a 2-site. 
	The same statement holds for $(\Cat(S)/X,J_X)$ and $(\Gpd(S)/X,J_X)$, \emph{mutatis mutandis}.

\end{proposition}

\begin{proof}
	That $J_X$ is closed under composition and contains identity arrows is immediate. 
	We can specify a strict pullback of a map $(w,a)\colon U\to Z$ in $J_X$ along $(f,b)\colon Y\to Z$, given a strict pullback 
	\begin{equation}\label{eq:str_pullback_for_lifting}
		\xymatrix{
			Y\times_Z U \ar[r]^{\tilde{f}} \ar[d]_{\tilde{w}} &U \ar[d]^w\\
			Y \ar[r]_f & Z
		}
	\end{equation}
	of $w$ along $f$, as follows: let the map $Y\times_Z U \to X$ be the composite $Y\times_Z U \xrightarrow{\tilde{w}} Y \to X$. 
	The map $Y\times_Z U \to Y$ in $\Cat(S)/_lX$ is $(\tilde{w},\id)$, and as $\tilde{w} \in J$ and $\id$ is invertible, this is in $J_X$ as required.
	The map $Y\times_Z U \to U$ is $(\tilde{f},c)$ where $c$ is
	\[
		\xymatrix{
		& U \ar[dr]^(.8){\ }="t1" \ar@/^1pc/[drr]_{\ }="s1"\\
		Y\times_Z U \ar[dr]_{\tilde{w}} \ar[ur]^{\tilde{f}} && Z \ar[r] & X\\
		& Y \ar[ur]_(.8){\ }="s2" \ar@/_1pc/[urr]^{\ }="t2"
		\ar@{=>}"s1";"t1"^{a^{-1}}
		\ar@{=>}"s2";"t2"^{b}
		}
	\]
	It is then easy to check that (\ref{eq:str_pullback_for_lifting}) lifts to a commuting square in $\Cat(S)/_lX$.

	To see that an arrow $(w,a)$ in $J_X$ is ff, we use the fact that given a diagram 
	
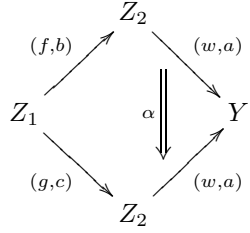
\[
		\xymatrix{
		&Z_2 \ar[dr]^{(w,a)}_(.4){\ }="s" \\
		Z_1 \ar[ur]^{(f,b)} \ar[dr]_{(g,c)} && Y\\
		& Z_2 \ar[ur]_{(w,a)}^(.4){\ }="t"
		\ar@{=>}"s";"t"_{\alpha}
		}
	\]
	in $\Cat(S)/_lX$, we can find a unique $\beta\colon f\Rightarrow g\colon Z_1 \to Z_2$ in $\Cat(S)$ such that $\id_w\circ \beta = \alpha$. 
	To see that $\beta$ lifts to a 2-arrow in $\Cat(S)/_lX$, we paste $a^{-1}$ and the 2-arrow
	
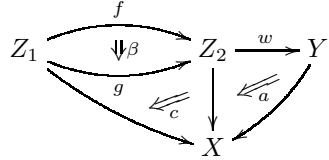
\[
		\xymatrix{
			Z_1 \ar@/^.8pc/[rr]^f_{\ }="s1" \ar@/_.8pc/[rr]_g^{\ }="t1" \ar@/_.4pc/[drr]^(.6){\ }="t2" && Z_2 \ar[r]^w \ar[d]_(.4){\ }="s2"^{\ }="t3" & Y \ar@/^.4pc/[dl]_(.2){\ }="s3"\\
			&& X
			\ar@{=>}"s1";"t1"^\beta
			\ar@{=>}"s2";"t2"^c
			\ar@{=>}"s3";"t3"^a
		}
	\]
	and get $b$, the required condition to give a 2-arrow in $\Cat(S)/_lX$. 
\end{proof}

\begin{corollary}
The 2-cateories $\Cat(S)/_lX$, $\Cat(S)/X$ and $\Gpd(S)/X$ admit bicategories of fractions for the classes $W_{J_X}$ of weak equivalences.
\end{corollary}

If we assume that $J$ satisfies the condition WISC from \cite{Roberts_12} (namely all slices $J/x$ have a weakly initial set), then $W_{J_X}$ has a locally small cofinal class; the localisations above are then locally essentially small.

These 2-categories are not examples of 2-categories of internal categories or groupoids in some 1-category, so are not covered by the results of \cite{Roberts_12}.

\begin{remark}
	Given a 2-site $(K,J)$, if every arrow $j\colon u\to x\in J$ is such that $j^*\colon K(x,z) \to K(u,z)$ is fully faithful, then one can construct a simpler model for the localisation $K[W_J^{-1}]$, where 2-arrows are no longer equivalence classes of diagrams, but given by individual diagrams.
	This condition holds for 2-sites of the form $(\Cat(S),J(T))$ where $T$ is a subcanonical singleton pretopology (as well as various sub-2-categories) \cite{Roberts_12}.
	This approach will be taken up in future work.
\end{remark}

Finally, note that nothing in the above relies on $K$ being a (2,1)-category, namely one with only invertible 2-arrows. 
This is usually assumed for results subsumed by theorem \ref{thm:2-sites_have_fractions}, but is unnecessary in the framework presented here.


\end{document}